\documentclass{amsart}

\usepackage{amsmath, enumerate, amsfonts, amssymb, amsthm, graphics, graphicx, verbatim, caption, subcaption}

\usepackage[english]{babel}
\usepackage[margin=0.95in]{geometry}
\usepackage{tikz, float}
\usetikzlibrary{shapes}
\tikzstyle{subgroup}=[scale=1]

\newtheorem{theorem}{Theorem}
\newtheorem{proposition}{Proposition}
\newtheorem{corollary}{Corollary}
\newtheorem{lemma}{Lemma}

\def\CD#1{\mathcal{CD}#1}

\title{\bf Chermak-Delgado Simple Groups}

\author{Ryan McCulloch}
\address{Department of Mathematics, University of Bridgeport, Bridgeport, CT 06604}
\email{rmccullo@bridgeport.edu}

\date{\today}

\begin{document}

\begin{abstract}
This paper provides the first steps in classifying the finite solvable groups having Property A, which is a property involving abelian normal subgroups.  We see that this classification is reduced to classifying the solvable Chermak-Delgado simple groups, which the author defines.  The author completes a classification of Chermak-Delgado simple groups under certain restrictions on the primes involved in the group order.
\end{abstract}

\maketitle

\section{Introduction and Motivation}

This paper classifies the groups of order $qp^k$ ($q,p$ primes) satisfying the property that every quotient over the centralizer of a non-trivial abelian normal subgroup has larger order than the subgroup.  To develop this classification, connections are made between groups having this property and the study of Chermak-Delgado lattices.  Ultimately, for such groups, the action of $G/C_G(N)$ on a minimal normal subgroup $N$ is determined, and this action greatly restricts possibilities for $p$ and $q$ (Theorem 3).  The primitive case in which $N=C_G(N)$ is Corollary 3.

In this paper all groups are assumed finite.  A group $G$ is said to have \textbf{Property A} if for every non-trivial abelian normal subgroup, $A$, of $G$, $|G/C_G(A)| > |A|$.  In other words, a group $G$ has Property A if for every non-trivial abelian normal subgroup $A$, of $G$, $G/C_G(A)$ embeds into $Aut(A)$ as a subgroup of size larger than $A$.   The nonabelian simple groups satisfy this property vacuously.  It is the author's ultimate goal to classify all of the solvable groups having Property A.

$S_4$ has Property A since the Klein $4$-group, $K_4$, is the unique non-trivial abelian normal subgroup of $S_4$, and $K_4$ is self-centralizing in $S_4$, and thus $|S_4/C_{S_4}(K_4)| = 6 > 4 = |K_4|$.  One may view $S_4$ as a semidirect product $[V]Aut(V)$ where $V = C_2 \times C_2$.  Through this viewpoint, one discovers many examples of groups with Property A.

A group $G$ is said to be \textbf{primitive} if $G$ contains a maximal subgroup $M$ so that ${core}_G(M) = 1$.  Baer\cite{Bae57} characterized primitive groups into 3 types.  See Doerk and Hawkes\cite{Doe92} Theorem A.15.2 for the classification and proof.  A type 1 primitive group has the structure of a semidirect $[V]H$ where $V$ is a vector space and $H$ is a subgroup of $Aut(V)$ which acts irreducibly on $V$.  Thus $V$ is the unique abelian normal subgroup of $G$ and $V$ is self-centralizing in $G$.  The type 1 primitive groups include all of the solvable primitive groups.

\begin{proposition} Suppose $G = [V_1]H_1 \times \dots \times [V_n]H_n$ is a direct product of type 1 primitive groups where each $|H_i| > |V_i|$.  Then $G$ has Property A.
\end{proposition}

\begin{proof} A non-trivial abelian normal subgroup, $A$, of $G$ has the form $A = X_1 \times \dots \times X_n$ where each $X_i$ is either $V_i$ or the identity; and at least one $X_j$ is nonidentity.  As each $V_i$ is self-centralizing in $[V_i]H_i$ and $|H_i| > |V_i|$, one easily sees that $|G/C_G(A)| > |A|$.
\end{proof}

This provides plenty of examples of groups having Property A.  Our goal of this section is to prove that a group $G$ has Property A if and only if $G$ is a direct product of Chermak-Delgado simple groups, which we shall define shortly.

Let $G$ be a group and $H \leq G$.  Then $m_G(H) = |H||C_G(H)|$ is the Chermak-Delgado measure of $H$ in $G$.  Let $m^*(G) = \max \{ m_G(H) \, | \, H \leq G \}$ and then define $\CD(G) = \{ H \leq G \, | \, m_G(H) = m^*(G) \}$.  The subgroup collection $\CD(G)$ forms a sublattice of the lattice of subgroups of $G$.  Furthermore, if $H,K \in \CD(G)$, then $\langle H,K \rangle = HK$, $C_G(H) \in \CD(G)$, $C_G(H \cap K) = C_G(H)C_G(K)$, and $C_G(C_G(H)) = H$.
 
This modular, self-dual sublattice of the lattice of subgroups of $G$ was first introduced in a paper by Chermak and Delgado\cite{Che89} in 1989.  Proofs of the properties of $\CD(G)$ stated in the previous paragraph are found in Section 1.G of Isaacs' Finite Group Theory\cite{Isa08}.  We will use these basic properties extensively throughout the proofs of this paper.

\begin{proposition} Suppose $M$ is the least element in $\CD(G)$.  Then $M$ is an abelian normal subgroup of $G$ and $Z(G) \leq M$.
\end{proposition}

\begin{proof} As $M$ is the least element in $\CD(G)$, and $C_G(M) \in \CD(G)$, $M \leq C_G(M)$ and so $M$ is abelian.  As $M$ is the least element in $\CD(G)$, and $\CD(G)$ is closed under the order preserving operation of conjugation, we see that $M \unlhd G$.  Finally, $C_G(M) \leq G$, and so $Z(G) \leq C_G(C_G(M)) = M$.
\end{proof}

We now consider direct products.  Given $H \leq G_1 \times \dots \times G_n$, one sees that $C_G(H) = C_G({\pi}_1(H) \times \dots \times {\pi}_n(H))$, where ${\pi}_i$ is the projection of $H$ on the $i$-th coordinate.  From here one sees that the Chermak-Delgado lattice of a direct product is the Cartesian product of the Chermak-Delgado lattices.  We state this below.

\begin{proposition} $\CD(G_1 \times \dots \times G_n) = \CD(G_1) \times \dots \times \CD(G_n)$.
\end{proposition}

Given $H,K \in \CD(G)$, we use the notation $H \prec K$ to mean that $H < K$ and there is no $R \in \CD(G)$ so that $H < R < K$.  The following result appeared in a paper by Brewster-Wilcox\cite{Bre12}.

\begin{proposition}
Let $H,K \in \CD(G)$ with $H \prec K$.  Then $H \unlhd K$.  
\end{proposition}

We are now in a position to prove the main result of this section.  A group $G$ is \textbf{Chermak-Delgado simple} if $\CD(G) = \{ 1, G \}$.  All of the nonabelian simple groups are Chermak-Delgado simple, and $S_4$ is an example of a solvable Chermak-Delgado simple group.  In fact any type 1 primitive group $[V]H$ with $|H| > |V|$ is Chermak-Delgado simple.

\begin{theorem} A group $G$ has Property A if and only if $G$ is a direct product of Chermak-Delgado simple groups.
\end{theorem}

\begin{proof} $\Leftarrow$) Suppose $G$ is a direct product of Chermak-Delgado simple groups.  From Proposition 3 we know that $\CD(G)$ is a Boolean lattice with $m^*(G) = |G|$, and the only abelian subgroup in $\CD(G)$ is $1$.  So for every non-trivial abelian normal subgroup $A$ of $G$, we must have $m_G(A) = |A||C_G(A)| < |G| = m^*(G)$, i.e. $|G/C_G(A)| > |A|$.  So $G$ has Property A.

$\Rightarrow$) Suppose $G$ has Property A.  By Proposition 2 we know that the least element, $M$, of $\CD(G)$ is an abelian normal subgroup of $G$.  Since $G$ has Property A, we know that every non-trivial abelian normal subgroup $A$ of $G$ has that $m_G(A) = |A||C_G(A)| < |G| = m_G(1)$.  So it must be that $M = 1 \in \CD(G)$ and $m^*(G) = |G|$.  If $G$ itself is Chermak-Delgado simple then we are done.  Suppose not and let $A_1,\dots A_k$ be all of the atoms in $\CD(G)$ (the atoms are the elements $A_i \in \CD(G)$ so that $1 \prec A_i$).  Note that each $A_i \unlhd G$.  This is true because $C_G(A_i) \prec G$ and so by Proposition 4, $C_G(A_i) \unlhd G$.  And since $A_i = C_G(C_G(A_i))$, it follows that $A_i \unlhd G$.  Since $G$ has Property A, it follows that none of the atoms are abelian.  Now for any $A_i \neq A_j$, we have that $A_i \cap A_j = 1$, and so $[A_i , A_j] = 1$.  From here it follows that for any atom $A_i$, $A_i \cap (A_1 \cdots A_{i-1}A_{i+1} \cdots A_k) = 1$.  This is true because otherwise, as $A_i$ is an atom, we would have that $A_i \leq A_1 \cdots A_{i-1}A_{i+1} \cdots A_k \leq C_G(A_i)$, contradicting $A_i$ being nonabelian.  So the product $A_1 \cdots A_k$ is a direct product.  Finally note that $C_G(A_1 \times \dots \times A_k) = 1$ as otherwise an atom is abelian.  And so $A_1 \times \dots \times A_k = C_G(C_G(A_1 \times \dots \times A_k)) = C_G(1) = G$.  Finally, each $A_i$ is itself a Chermak-Delgado simple group as otherwise we would have an element of $\CD(G)$ between $1$ and $A_i$.
\end{proof}

A group $G$ is said to be \textbf{indecomposable} if $G$ cannot be written as a direct product $H \times K$ with $H \neq 1$ and $K \neq 1$.  A Chermak-Delgado simple group is clearly indecomposable by Proposition 3, and so we get the following corollary:

\begin{corollary} An indecomposable group $G$ has Property $A$ if and only if $G$ is a Chermak-Delgado simple group.
\end{corollary}

We illustrate the power of the above corollary by restating it using the definitions of Property A and Chermak-Delgado simple:

\begin{proposition} If an indecomposable group $G$ has the property that for every $1 < A \leq G$ with $A$ abelian and normal, $|G : C_G(A)| > |A|$, then $G$ has the property that for every $1 < H < G$, $|G : C_G(H)| > |H|$.
\end{proposition}

Theorem 1 motivates a classification of Chermak-Delgado simple groups.  Of course the classification of nonabelian simple groups is a special case of the classification of all Chermak-Delgado simple groups.  Our focus is a classification of the solvable Chermak-Delgado simple groups, and we provide the first steps in that classification in the next section.

\section{Chermak-Delgado Simple Groups}

One of the first applications of Sylow theory is the ruling out of potential group orders for simple groups.  We begin the same way in our study of Chermak-Delgado simple groups.  It turns out that the identity subgroup and $S_4$ are the only Chermak-Delgado simple groups having orders between 1 and 50.  There are fast and fun ways to rule out all other orders.  For example, groups of order 36 or 45 will have a Sylow $9$-subgroup, which must be abelian, and hence has Chermak-Delgado measure at least 81, which is larger than the group order.  Order 40 is a bit tricky to rule out.  The Sylow $5$-subgroup is normal, and so a Sylow $2$-subgroup acts on it.  A group of order 5 has automorphism group of order 4, so there is a non-trivial kernel to this action.  Thus the Sylow $5$-subgroup will have measure at least 50, which is larger than 40.  The hardest to rule out is 48.  It is true that any group of order 16 has an abelian subgroup of order 8, and so the abelian subgroup will have measure at least 64, which is larder than 48.

We'd like to eliminate a large chunk of orders at once.  Suppose $G$ is a group  having one ``very large prime'' in its group order.  Formally, suppose $|G| = mp^k$ with $p$ a prime, $k \geq 1$, and $m < p$.  Then a Sylow $p$-subgroup, $P$, of $G$, will have that $m_G(P) > |G|$ as $Z(P) > 1$.  So groups of such order cannot be Chermak-Delgado simple.  We wondered if we could do any better than that.  

Suppose $G$ is a group having one ``very large prime'' $p$ in its group order, where $p$ is ``not too large''.  Formally, suppose $|G| = mp^k$ with $p$ a prime, $k \geq 1$, and $m$ a natural number so that $m/q < p < m$ where $q$ is the smallest prime divisor of $m$.  

This is the scenario that we will focus on for the rest of this paper.  It splits into two cases.  The first of which is when $1 < m/q$.  In that case we establish that a group, $G$, of such order cannot be Chermak-Delgado simple.  The other case is when $m=q$, i.e. $|G| = qp^k$ with $q < p$.  In that case we classify all Chermak-Delgado simple groups of such order.  Note that $S_4$ is included in this classification as $24=3 \cdot 2^3$.

\subsection{The case when $1 < m/q$.}

\begin{proposition}\label{prop2.1}
Suppose $G$ is a finite group and $m^*(G)=|G|$.  Then $G$ contains no non-trivial, normal, cyclic subgroups.
\end{proposition}

\begin{proof}
Suppose $\langle x \rangle \unlhd G$ with $x \in G$ nonidentity.  It follows from the orbit-stabilizer theorem that $|G| = |x^G||C_G(x)|$.  Notice that $x^G$ is properly contained in $\langle x \rangle$, as $\langle x \rangle \unlhd G$ and $ 1 \notin x^G$.  So $m_G(\langle x \rangle) = |\langle x \rangle||C_G(\langle x \rangle)| > |x^G||C_G(x)| = |G| = m^*(G)$, a contradiction.
\end{proof}

\begin{lemma}\label{lem2.4}
Suppose $G$ is a group with $|G| = mp^k$ with $p$ a prime, $k \geq 1$, and $m$ a natural number so that $m/q < p < m$ where $q$ is the smallest prime divisor of $m$.  Then a Sylow $p$-subgroup of $G$ is either normal or self-normalizing.
\end{lemma}

\begin{proof}
Suppose a Sylow $p$-subgroup of $G$ is not normal.  Let $n_p$ denote the number of Sylow $p$-subgroups of $G$.  Then $n_p > p$ and $n_p$ divides $m$.  As $m/q < p < m$ and $q$ is the smallest prime divisor of $m$, it follows that $n_p = m$, and so a Sylow $p$-subgroup of $G$ is self-normalizing.
\end{proof}

The next lemma will be applied to the case when $k=1$ in Theorem~\ref{thm2.3}.

\begin{lemma}\label{lem2.3}
Suppose $G$ is a group with $|G| = mp$ with $p$ a prime and $m$ a natural number so that $m/q < p < m$ where $q$ is the smallest prime divisor of $m$.  Suppose that a Sylow $p$-subgroup of $G$ is not normal.  Then $G$ possesses a characteristic, abelian $p$-complement.
\end{lemma}

\begin{proof}
A Sylow $p$-subgroup of $G$ is cyclic of order $p$, and so any two distinct Sylow $p$-subgroups of $G$ have trivial intersection.  By Lemma~\ref{lem2.4}, $|{Syl}_p(G)| = |G:P| = m$, where $P \in {Syl}_p(G)$.  Let $X$ be the set of all elements of $G$ that are not conjugate in $G$ to any nonidentity element in $P$.  Then $|X| = |G| - (|G:P|)(|P| - 1) = mp - m(p-1) = m$ and the set $X$ is invariant under any automorphism of $G$.  Let $1 \neq x \in X$ be arbitrary.  Now $|C_G(x)|$ is coprime to $p$, as otherwise $x$ centralizes a Sylow $p$-subgroup of $G$ which is self-normalizing, a contradiction.   As $X$ contains all of the elements of $G$ of order coprime to $p$, we have that $C_G(x) \subseteq X$.  We show that $C_G(x) = X$.  Otherwise, if $|C_G(x)| < m$, then $|x^G| = rp$ for some $r \geq q$.  One then sees that $|x^G| = rp \geq qp > m$, which is a contradiction as $x^G \subseteq X$ and $|X| = m$.  So $C_G(x) = X$, and $X$ is a subgroup.  So $X$ is a characteristic $p$-complement for $G$, and $C_G(x) = X$ was true for arbitrary $1 \neq x \in X$, and so $X$ is abelian.
\end{proof}

For the next lemma it is essential that $1 < m/q$.

\begin{lemma} \label{lem2.6}
Suppose $G$ is a group with $|G| = mp^k$ with $p$ a prime, $k \geq 1$, and $m$ a natural number so that $1 < m/q < p < m$ where $q$ is the smallest prime divisor of $m$.  If $P \neq R \in {Syl}_p(G)$, then $G = \langle P , R \rangle$, and $P \cap R = O_p(G)$ has order $p^{k-1}$.
\end{lemma}

\begin{proof}
As $1 < m/q$ and $q$ is the smallest prime divisor of $m$, we have that $q \leq m/q < p$.  So $q$ is the smallest prime divisor of $|G|$ and $p$ is the largest prime divisor of $|G|$.

Suppose $P \neq R \in {Syl}_p(G)$.  Notice that the product $PR$ has cardinality equal to $p^{2k} / |P \cap R|$.  If $|P \cap R| < p^{k-1}$, then $|PR| \geq p^2 p^k > p q p^k > mp^k = |G|$.  This is a contradiction as $PR$ is a subset of $G$.  So $|P \cap R| = p^{k-1}$ and $|PR| = p^{k+1}$.  

Consider the join $\langle P,R \rangle$.  The product $PR \subset \langle P,R \rangle$, and so $|\langle P,R \rangle| > p^{k+1}$.  Observe that $p^k$ divides $|\langle P,R \rangle|$ which divides $mp^k$.  Now as $m/q < p < m$ where $q$ is the smallest prime divisor of $m$, it must be that $|\langle P,R \rangle| = mp^k$, and so $\langle P,R \rangle = G$.  Now $P \cap R$ has index $p$ in both $P$ and $R$, and so $P \cap R \unlhd P$ and $P \cap R \unlhd R$, and so $P \cap R \unlhd \langle P,R \rangle = G$.  So $O_p(G) = P \cap R$. 
\end{proof}

The direct product of the Frobenius group of order $56$ with any $7$-group is a nice example to have in mind when thinking through Lemma~\ref{lem2.6}.

This next lemma is a consequence of Dickson's Theorem\cite{Dic01} which classifies the subgroups of $PSL(2,q)$ where $q$ a power of a prime.  See Huppert\cite{Hup67} for a good presentation of Dickson's Theorem.  However, one does not need Dickson's Theorem to prove Lemma~\ref{lem2.7}.  We state the lemma here and leave an elementary proof to the reader.

\begin{lemma}\label{lem2.7}
Suppose $q$ is a power of a prime $p$.  Then $GL(2,q)$ contains $q+1$ Sylow $p$-subgroups.  If $S \neq T \in {Syl}_p(GL(2,q))$, then $S \cap T = 1$ and $\langle S , T \rangle = SL(2,q)$.  
\end{lemma}

This leads to the following corollary:

\begin{corollary} \label{cor2.5}
Suppose $q$ is a power of a prime $p$.  Suppose $H$ is a subgroup of $GL(2,q)$ so that $q$ divides $|H|$.  Then either $H$ has a normal Sylow $p$-subgroup, or $H$ contains $SL(2,q)$.
\end{corollary}

\begin{proof}
As $q$ divides $|H|$, a Sylow $p$-subgroup of $H$ is also a Sylow $p$-subgroup of $GL(2,q)$, and so if $H$ contains at least two distinct Sylow $p$-subgroups of $GL(2,q)$, then by Lemma~\ref{lem2.7}, $H$ contains $SL(2,q)$.
\end{proof}

Now we are ready to prove our main theorem.

\begin{theorem} \label{thm2.3}
Suppose $G$ is a group with $|G| = mp^k$ with $p$ a prime, $k \geq 1$, and $m$ a natural number so that $1 < m/q < p < m$ where $q$ is the smallest prime divisor of $m$.  Then $m^*(G) > |G|$ and so $G$ is not Chermak-Delgado simple.
\end{theorem}

\begin{proof}
Suppose by way of contradiction that $m^*(G) = |G|$.  Let $P \in {Syl}_p(G)$.  If $|Z(P)| > p$, then $|P||Z(P)| \geq p^kp^2 > p^kqp > mp^k = |G|$, a contradiction.  So $Z(P)$ is cyclic of order $p$.  If $P \unlhd G$, then $Z(P) \unlhd G$, contrary to  Proposition~\ref{prop2.1}.  So $P$ is not normal in $G$.

Let $P \neq R \in {Syl}_p(G)$.  Then by Lemma~\ref{lem2.6}, $P \cap R = O_p(G)$ has order $p^{k-1}$.  If $O_p(G) = 1$, then $k=1$, and so by Lemma~\ref{lem2.3}, $G$ contains an abelian $p$-complement, $M$, and $|M|^2 = m^2 > pm = |G|$, a contradiction.  So $O_p(G) > 1$.  By Proposition~\ref{prop2.1} we know that $|Z(O_p(G))| > p$.  Suppose $|Z(O_p(G))| > p^2$.  Then $|O_p(G)||Z(O_p(G))| \geq p^{k-1}p^3 = p^kp^2 > p^kqp > mp^k = |G|$, a contradiction.  So $|Z(O_p(G))| = p^2$.  By Proposition~\ref{prop2.1}, $Z(O_p(G))$ cannot be cyclic, so $Z(O_p(G))$ is elementary abelian.  Finally we show that $C_G(Z(O_p(G))) = O_p(G)$.  Otherwise, if $C_G(Z(O_p(G))) > O_p(G)$, then $|C_G(Z(O_p(G)))| \geq qp^{k-1}$.  However, now, $|Z(O_p(G))||C_G(Z(O_p(G)))| \geq p^2 qp^{k-1} = p^kqp > mp^k = |G|$, a contradiction.

So $G/C_G(Z(O_p(G))) = G/O_p(G) = H$ is isomorphic to a subgroup of $GL(2,p)$, and we will identify $H$ with the corresponding subgroup of $GL(2,p)$.  Now a Sylow $p$-subgroup of $H$ is not normal (as $H = G/O_p(G)$), and so by Corollary~\ref{cor2.5}, $H$ contains $SL(2,p)$, which has order $p(p+1)(p-1)$, and $H$ contains all $p+1$ of the Sylow $p$-subgroups of $GL(2,p)$.  We know from Lemma~\ref{lem2.4} that $m$ is the number of Sylow $p$-subgroups of $H$, and so $|H| = (p+1)p$.  So $p-1=1$, i.e. $p=2$, contradicting the fact that $q < p$.
\end{proof}

\subsection{The case when $m=q$.}

We are now considering groups of order $qp^k$ where $p$,$q$ are primes with $p<q$.  Theorem 1.36 in \cite{Isa08} contains a slick, Sylow theory proof that groups of such order are solvable.  The result dates back to the 19th century.  It appears in Burnside's Theory of Groups of Finite Order\cite{Bur97} from 1897.  Of course it is a special case of Burnside's  $p^{\alpha}q^{\beta}$ Theorem\cite{Bur04}, which was proven shortly afterwards using character theory.  We state it here as our Proposition~\ref{prop2.3}.

\begin{proposition} \label{prop2.3}
Suppose $G$ is a group with $|G| = qp^k$ with $p,q$ primes and $k \geq 1$.  Then $G$ is solvable.
\end{proposition}

Note that in the above proposition, when $p \geq q$, $G$ is clearly solvable as a Sylow $p$-subgroup of $G$ is normal.  The case when $p < q$ is the tricky case.  

Now a lemma:

\begin{lemma} \label{lem2.5}
Suppose $V$ is an $n$-dimensional vector space over $GF(p)$, and suppose $K$ is a subgroup of $Aut(V)$.  If $K$ acts irreducibly on $V$, then $O_p(K) = 1$.  If, furthermore, $|K| = qp^r$ with $p < q$ a prime and $r \geq 1$, then a Sylow $q$-subgroup of $K$ is normal in $K$.
\end{lemma}

\begin{proof}
Suppose that $K$ acts irreducibly on $V$.  Note that the points of $V$ left fixed by $O_p(K)$ form a $K$-invariant subgroup of $V$.  By the orbit-stabilizer theorem, $O_p(K)$ has non-trivial fixed points, and as $K$ is irreducible, $O_p(K)$ must fix all of $V$, and as $O_p(K)$ acts faithfully on $V$, it must be that $O_p(K) = 1$.  If, furthermore, $|K| = qp^r$ with $p < q$ a prime and $r \geq 1$, then by Proposition~\ref{prop2.3} $K$ is solvable, and so a Sylow $q$-subgroup of $K$ is normal in $K$.  
\end{proof}

Here is how the Chermak-Delgado measure condition plays in:

\begin{proposition} \label{prop2.12}
Suppose $G$ is a group with $|G| = qp^k$ with $p < q$ primes and $k \geq 1$, and suppose that $m^*(G) = |G|$.  Suppose $N$ is a minimal normal subgroup of $G$.  Then $|N| = p^n$ for some $n \geq 1$ and $|G/C_G(N)| = qp^r$ for some $r \geq 1$ with $p^n < qp^r$.  Furthermore, a Sylow $q$-subgroup of $G/C_G(N)$ acts irreducibly on $N$.
\end{proposition} 

\begin{proof}
If $Q \in {Syl}_q(G)$ with $Q \unlhd G$, then given $P \in {Syl}_p(G)$, $P$ acts on $Q$ via conjugation, and the kernel of that action, $K$, is a normal subgroup of $P$.  If $K \neq 1$, then $K \cap Z(P) \neq 1$, but $K \cap Z(P) \leq Z(G) = 1$, a contradiction.  So $K = 1$ and so $P$ embeds in $Aut(Q)$ which has order $q-1$, and so $p^k < q$.  However, then $m_G(Q) = q^2 > qp^k = |G|$, a contradiction.  

So $O_q(G) = 1$.  Now by Proposition~\ref{prop2.3}, $G$ is solvable, and so a minimal normal subgroup, $N$, of $G$ has order $p^n$ for some $n \geq 1$.  As $N$ is minimal normal, $G/C_G(N)$ acts irreducibly and faithfully on $N$ which can be viewed as an $n$-dimensional vector space over $GF(p)$.  So by Lemma~\ref{lem2.5}, $O_p(G/C_G(N)) = 1$.  

Observe that $G/C_G(N)$ cannot be cyclic of order $q$, as otherwise $q$ divides $|Aut(N)|$ and since $q$ is a prime, we have $q < p^n$.  However, then $m_G(N) = |N||C_G(N)| = p^n p^k > q p^k = |G|$, a contradiction.  So $|G/C_G(N)| = qp^r$ for some $r \geq 1$.  We see that $|N| = p^n < qp^r = |G/C_G(N)|$.

Denote $G/C_G(N) = H$.  By Lemma~\ref{lem2.5}, $T \in {Syl}_q(H)$ is normal in $H$ and $O_p(H) = 1$.  And so $S \in {Syl}_p(H)$ acts by conjugation faithfully on $T$ (the kernel of that action is normal in $H$).  So $S$ embeds in $Aut(T)$ which has order $q-1$ and so $p^r < q$.  

Now we argue that $T$ acts irreducibly on $N$.  Suppose not.  $T$ cannot act trivially on all of $N$ as $T$ is a subgroup of $G/C_G(N)$ which acts faithfully.  As $T$ has prime order, either $T$ acts faithfully on a subgroup of $N$, or $T$ acts trivially.   Since $T$ is normal in $H$, $T$ cannot have non-trivial, proper fixed points in $N$, as all of the points of $N$ fixed by $T$ form a normal subgroup of $G$, contradicting the minimality of $N$.  By Maschke's Theorem we properly decompose $N = W_1 \times W_2$ where $T$ acts faithfully on each $W_i$.  Let $d_1 = \dim(W_1)$ and $d_2 = \dim(W_2)$ and let $d = \min\{d_1,d_2\}$.  Since $q$ is a prime, $q < p^d$.  However, now $p^r < q < p^d$, and so $qp^r < qp^d < p^{2d} \leq p^{d_1 + d_2} = p^n$.  So $qp^r = |G/C_G(N)| < |N| = p^n$, and so $m_G(N) = |N||C_G(N)| > |G|$, a contradiction.  So $T$ acts irreducibly on $N$.  
\end{proof}

This next proposition is Lemma 6.3 in Gorenstein's Finite Groups\cite{Gor80}.   

\begin{proposition} \label{prop2.11}
Suppose $T \in GL(n,p)$ has order $m$ and suppose $T$ is irreducible.  Then $m \, | \, p^n - 1$ and $m \nmid p^d - 1$ for any $d < n$.
\end{proposition}

In $GL(n,p)$ the irreducible cyclic subgroups of order $p^n - 1$ are called the \textbf{Singer cycles} (in fact this definition also extends to $GL(n,q)$ where $q$ is a power of a prime $p$).  The following is contained in Theorem II.7.3a on pg.187 in Huppert\cite{Hup67}, and we state it the way it appears in Short\cite{Sho92}.

\begin{proposition} \label{prop2.8}
Let $T$ be an irreducible subgroup of a Singer cycle $S$ of $GL(n,p)$.  Then the normalizer in $GL(n,p)$ of $T$ is the semidirect product of $S$ and a cyclic group $L$ of order $n$.  
\end{proposition}

Let $p$ be a prime and $n \geq 1$.  To get a Singer cycle, start by fixing a representation of an $n$-dimensional vector space, $V$, over $GF(p)$ as the additive group of $GF(p^n)$.  Note that each element $t \in V / \{ 0 \}$ induces an automorphism ${\sigma}_t$ of $V$, defined by $(v){\sigma}_t = tv$, where the multiplication $tv$ is in the field $GF(p^n)$.  Let $S$ be the subgroup of $Aut(V)$ consisting of all such ${\sigma}_t$.  $S$ is canonically isomorphic to the multiplicative group of units of $GF(p^n)$, and thus is cyclic of order $p^n -1$, and so $S$ is a Singer cycle.  Let $L$ be the Galois group of $GF(p^n)$ over $GF(p)$ and view $L$ as a subgroup of $Aut(V)$.  So $L$ is cyclic of order $n$, and is generated by the Frobenius map $\varphi$ defined as $v^{\varphi} = v^p$, raising to the power $p$.  Given any subgroup $T$ of $S$, one can see that $L$ normalizes $T$: $(v){{\sigma}_t}^{{\varphi}} = t^{p} v$ for any $v \in V$, and thus ${{\sigma}_t}^{{\varphi}} = {\sigma}_{t^{p}}$.  For the remaining results in this section, we retain the names $S$ and $L$ for the specific groups described above.

The following follows rather easily from Proposition~\ref{prop2.8}:

\begin{proposition} \label{prop2.13}
Suppose $V$ is an $n$-dimensional vector space over $GF(p)$, and fix a representation of $V$ as the additive group of $GF(p^n)$; and suppose $K$ is a subgroup of $Aut(V)$ so that $|K| = qp^r$ with $p < q$ a prime and $r \geq 1$.  Suppose that a Sylow $q$-subgroup of $K$ acts irreducibly on $V$.  Then $K$ is conjugate in $Aut(V)$ to the automorphism group $TH$ with $T \leq S$ of order $q$ and $H \leq L$ of order $p^r$, where $S$ and $L$ are defined as in the previous paragraph.  It follows that $p^r \, | \, n$.
\end{proposition}

\begin{proof}
We will view $V$ as the additive group of $GF(p^n)$.  Let $Q \in {Syl}_q(K)$.  Now $Q$ is irreducible and cyclic of order $q$ a prime, so by Proposition~\ref{prop2.11} we have that $q$ divides $p^n - 1$ and $q$ does not divide $q^d - 1$ for $d < n$.  And so a Singer cycle contains a Sylow $q$-subgroup of $Aut(V)$.  $Q$ is contained in a Sylow $q$-subgroup of $Aut(V)$, and so $Q^x = T$ for $T \leq S$, where $S$ is the Singer cycle described earlier, and $T$ is the subgroup of $S$ of order $q$, and for some $x \in Aut(V)$.  As $Q$ acts irreducibly on $V$, surely $K$ acts irreducibly on $V$, and so by Lemma~\ref{lem2.5}, $O_p(K) = 1$ and $Q \unlhd K$.  Let $P \in {Syl}_p(K)$.  Then $P \leq N_{Aut(V)}(Q)$ and so $P^x \leq N_{Aut(V)}(Q^x) = N_{Aut(V)}(T) = SL$, where $L$ is the Galois group of $GF(p^n)$; this is by Proposition~\ref{prop2.8}.  Now $|P^x|$ is coprime to $|S|$, and so $P^x \leq R^{\sigma}$, where $R \leq L$ is a Sylow $p$-subgroup of $L$ for some $\sigma \in S$.  So $(P^x)^{{\sigma}^{-1}} = H \leq L$, where $H \leq L$ is of order $p^r$.   We see that $K^{x{\sigma}^{-1}} = {(QP)}^{x{\sigma}^{-1}} = Q^xP^{x{\sigma}^{-1}} = TH$.  
\end{proof}

Now that we know that $p^r \, | \, n$, our Chermak-Delgado measure condition leaves few possibilities:

\begin{lemma} \label{lem2.10}
Suppose $p,q$ are primes with $p^n < qp^r$ for $n,r \geq 1$, and suppose that $q \, | \, p^n - 1$ and $p^r \, | \, n$.  Then we have exactly one of the following:
\begin{enumerate}
\item $p=2$, $q=3$, $n=2$ and $r=1$.
\item $p=2$, $q=5$, $n=4$ and $r=2$.
\item $p > 2$, $q=(p^p - 1)/(p-1)$, $n=p$ and $r=1$.
\end{enumerate}
\end{lemma}

\begin{proof}
First suppose $p=2$.  As $2^r \, | \, n$, $n$ must be even.  

Suppose by way of contradiction that $n > 4$.  Then $2^n - 1 = (2^{n/2} - 1)(2^{n/2} + 1)$.  Since $q$ is a prime, $q$ divides one of these integer factors, and so $q \leq 2^{n/2} + 1$.  Note that $n <  2^{n/2} - 1$.  This can be shown by induction.  When $n=6$ the result is true.  Assume that $n <  2^{n/2} - 1$.  So $2^{n/2} > n + 1$.  So $2^{(n+2)/2} = 2(2^{n/2}) > 2(n+1) = n + n + 2 > (n + 2) + 1$.  So $(n+2) < 2^{(n+2)/2} - 1$.  So $q2^r \leq qn < (2^{n/2} + 1)(2^{n/2} - 1) = 2^n - 1 < 2^n$, which is false.  So $n \leq 4$.  

Checking $n=2$ and $n=4$ one sees that $q$ must be $3$ and $5$ respectively, and $r=1$ and $2$ respectively, and in both of these cases $2^n < q 2^r$.

Suppose $p > 2$.  Suppose by way of contradiction that $n = p^rk$ for $k > 1$.  Then $p^{p^rk} - 1 = (p^{p^r} - 1)(p^{p^r(k-1)} + p^{p^r(k-2)} + ... + p^{p^r} + 1)$.  Now as $k > 1$, $p^{p^r(k-1)} + p^{p^r(k-2)} + ... + p^{p^r} + 1 > p^{p^r} - 1$.  Since $q$ is a prime, $q$ divides one of these integer factors, and so $q \leq p^{p^r(k-1)} + p^{p^r(k-2)} + ... + p^{p^r} + 1$.  Clearly $p^r < p^{p^r} - 1$ as $r < p^r$.  So $qp^r < p^n$, a contradiction.  So $n = p^r$.

Suppose by way of contradiction that $r > 1$.  Now $p^{p^r} - 1 = (p-1)(p^{p(p-1)} + ... + p^p + 1)(p^{p^2(p-1)} + ... + p^{p^2} + 1)\cdots(p^{p^{r-2}(p-1)} + ... + p^{p^{r-2}} + 1)(p^{p^{r-1}(p-1)} + ... + p^{p^{r-1}} + 1)$.  Since $q$ is a prime, $q$ divides one of these integer factors, and so $q \leq p^{p^{r-1}(p-1)} + ... + p^{p^{r-1}} + 1$.  Now it is clear that for $r \geq 4$, $r < 3^{r-2} \leq p^{r-2}$.  So $p^r < p^{p^{r-2}} < p^{p^{r-2}(p-1)} + ... + p^{p^{r-2}} + 1$.  When $r=3$, $p^3 \leq p^p <  p^{p(p-1)} + ... + p^p + 1$.  When $r=2$, also $p^2 \leq p^{p-1} < p^{p-1} + ... + p + 1$.  So we have $qp^r < p^{p^r} - 1 < p^{p^r}$, a contradiction.

So $r=1$ and $n=p$.  Now $p^p - 1 = (p-1)(p^{p-1} + ... + p + 1)$.  Now $p^{p-1} + ... + p + 1 > p$ so $q$ cannot divide $p-1$ as otherwise $qp < p^p - 1 < p^p$.  So $q$ divides $p^{p-1} + ... + p + 1$.  If $q$ were to properly divide $p^{p-1} + ... + p + 1$, then $p^{p-1} + ... + p + 1 = qs$ for some $s > 1$.  So $pq = (p/s) (p^{p-1} + ... + p + 1) < (p-1)(p^{p-1} + ... + p + 1) = p^p - 1 < p^p$, a contradiction.  So $q = p^{p-1} + ... + p + 1 = (p^p - 1)/(p-1)$, and indeed $qp > q(p-1) + 1 = p^p$.
\end{proof}

Wagstaff\cite{Wag96} shows that the primes $p<180$ for which $(p^p - 1)/(p-1)=q$ is a prime are $p=2,3,19,31$.  

Putting all of this together we get the following theorem:

\begin{theorem} \label{thm2.4}
Suppose $G$ is a group with $|G| = qp^k$ with $p < q$ primes and $k \geq 1$, and suppose that $m^*(G) = |G|$.  Suppose $N$ is a minimal normal subgroup of $G$.  Then $|N| = p^n$ for some $n \geq 1$ and $|G/C_G(N)| = qp^r$ for some $r \geq 1$.  View $N$ as an $n$-dimensional vector space over $GF(p)$ and fix a representation of $N$ as the additive group of $GF(p^n)$; and view $G/C_G(N)$ as a subgroup of $Aut(N)$.  Then $G/C_G(N)$ is conjugate in $Aut(N)$ to the automorphism group $TH$ with $T \leq S$ of order $q$ and $H \leq L$ of order $p^r$, where $S$ and $L$ are defined as before.  Furthermore, we have exactly one of the following:
\begin{enumerate}
\item $N$ is $2$-dimensional over $GF(2)$, $T$ has order $3$, and $H$ has order $2$ (and so $TH \cong S_3$).
\item $N$ is $4$-dimensional over $GF(2)$, $T$ has order $5$, and $H$ has order $4$.
\item $N$ is $p$-dimensional over $GF(p)$ with $p > 2$, $T$ has order $(p^p - 1)/(p-1) = q$, and $H$ has order $p$.
\end{enumerate}
\end{theorem}

The case in which the solvable group $G$ is primitive is worth noting as a corollary.

\begin{corollary} \label{cor2.4}
Suppose $G$ is a primitive group with $|G| = qp^k$ with $p < q$ primes and $k \geq 1$, and suppose that $m^*(G) = |G|$.  Then $G \cong [V]K$ where $V$ is the additive group of $GF(p^n)$ viewed as an $n$-dimensional vector space over $GF(p)$, and $K = TH$ with $T \leq S$ of order $q$ and $H \leq L$ of order $p^r$, $r \geq 1$, where $S$ and $L$ are defined as before.  Furthermore, we have exactly one of the following:
\begin{enumerate}
\item $V$ is $2$-dimensional over $GF(2)$, $T$ has order $3$, and $H$ has order $2$ (and so $G \cong S_4$).
\item $V$ is $4$-dimensional over $GF(2)$, $T$ has order $5$, and $H$ has order $4$.
\item $V$ is $p$-dimensional over $GF(p)$ with $p > 2$, $T$ has order $(p^p - 1)/(p-1) = q$, and $H$ has order $p$.
\end{enumerate}
\end{corollary}

We note that one can construct the group $G = V[K]$ above with $K$ acting faithfully and irreducibly on $V$, and for each choice of primes above, one sees that $G$ is Chermak-Delgado simple since $|K|>|V|$.

\section{Conclusions}
Section 1 establishes the equivalence of the indecomposable groups having Property A with the Chermak-Delgado simple groups.  Theorem 2 shows that no Chermak-Delgado simple groups exist having order $mp^k$  with $p$ a prime, $k \geq 1$, and $m$ a natural number so that $1 < m/q < p < m$ where $q$ is the smallest prime divisor of $m$.  Theorem 3 provides conditions on a Chermak-Delgado simple group of order $qp^k$ with $p < q$ primes in terms of the action of $G/C_G(N)$ on a minimal normal subgroup $N$.  Corollary 3 restates these conditions for the primitive case where $G$ splits as $[N]G/C_G(N)$.  In that case we construct primitive, Chermak Delgado simple examples for all allowable values of the primes $p$ and $q$.  $S_4$ is included as the smallest allowable example.

This is only a small step towards classifying the solvable groups having Property A, and it seems to be a worthwhile endeavor.  The next step would be a classification of the indecomposable groups having Property A of order $p^{\alpha}q^{\beta}$ (i.e. a classification of the Chermak-Delgado simple groups of order $p^{\alpha}q^{\beta}$).  On a different note, the author has constructed examples of indecomposable groups $G$ with $m^*(G) = |G|$ which are not Chermak-Delgado simple, and these groups will be discussed in a later paper.  Most of the work in this paper is contained in the author's Ph.D. dissertation, and the author would like to thank his advisor Ben Brewster for his guidance and support during this adventure.

\bibliographystyle{amsplain}
\bibliography{references}

\end{document}